%% file: a18.tex
\documentclass{amsart}
\input{command}
\theorems
\def\vir{\mathrm{vir}}
\def\re{{\rm{re}}}
\def\im{{\rm{im}}}
\def\Deo{\dot\De}
\def\Deop{\Deo_+}
\def\Deom{\Deo_-}

\def\qt{\mat{\mathbb T}}

\oper{CM}
\def\lvc{\wtl\lM_\cC}

\usepackage{hyperref}
\remfalse

\begin{document}
\input{title}
%\tableofcontents
\input{intro.tex}

\input{prelim.tex}

\input{s01.tex}

\input{s02.tex}
\input{s03.tex}

\input{s04.tex}

\bibliography{../tex/papers}
\bibliographystyle{../tex/hamsplain}

\end{document}

%% file: title.tex
\title[Motivic DT invariants and McKay correspondence]{Motivic Donaldson-Thomas invariants and McKay correspondence}%
\author{Sergey Mozgovoy}%
\email{mozgovoy@maths.ox.ac.uk}%
%\date{\today}
\begin{abstract}
Let $G\sb\SL_2(\cC)\sb\SL_3(\cC)$ be a finite group. We compute motivic Pandharipande-Thomas and Donaldson-Thomas invariants of the crepant resolution $\Hilb^G(\cC^3)$ of $\cC^3/G$ generalizing results of Gholampour and Jiang who computed numerical DT/PT invariants using localization techniques. Our formulas rely on the computation of motivic Donaldson-Thomas invariants for a special class of quivers with potentials. We show that these motivic Donaldson-Thomas invariants are closely related to the polynomials counting absolutely indecomposable quiver representations over finite fields introduced by Kac. We formulate a conjecture on the positivity of Donaldson-Thomas invariants for a broad class of quivers with potentials. This conjecture, if true, implies the Kac positivity conjecture for arbitrary quivers.
\end{abstract}

\maketitle

%% file: intro.tex
\section{Introduction}
The goal of this paper is to compute motivic Pandharipande-Thomas and\br Donaldson-Thomas invariants of the crepant resolution $Y=\Hilb^G(\cC^3)$ of $\cC^3/G$ for any finite subgroup $G\sb\SL_2(\cC)\sb\SL_3(\cC)$. This is achieved by using the idea of Nagao and Nakajima \cite{nagao_counting} who realized PT and DT moduli spaces on $Y$ as moduli spaces of stable framed representations of a certain quiver with potential for particular choices of stability parameters.

The quiver $\what Q$ considered above is the McKay quiver of $(G,\cC^3)$. It can be constructed as follows \cite[Figure 1]{gholampour_counting}. One starts with a quiver $Q$ of affine type, considers its double quiver $\ub Q$ and then adds loops $l_i:i\to i$ at every vertex $i\in Q_0$. Potential on $\what Q$ mentioned above is given by
$$W=\sum_{(a:i\to j)\in Q_1}(aa^*l_j-a^*al_i),$$
where $(a^*:j\to i)\in\ub Q_1$ is an arrow dual to $(a:i\to j)\in Q_1$.
This construction works of course for an arbitrary quiver $Q$. Now we can ask, what are the motivic Donaldson-Thomas invariants of $(\what Q,W)$. The answer is given by the following result (see Theorem \ref{u motive}).

\begin{thr}
\label{main}
Let $Q$ be an arbitrary quiver and let $(\what Q,W)$ be constructed as above.
%Let $\Pi_Q$ be the preprojective algebra of $Q$ (quotient of $\cC\ub Q$ by $\sum_{a\in Q_1}(aa^*-a^*a)$).
For any $\al\in\cN^{Q_0}$ let $a_\al(q)$ be the polynomial counting absolutely indecomposable representations of $Q$ with dimension vector \al over finite fields \cite{kac_root}. 
Then the universal motivic DT series of the Jacobian algebra $J_{\what Q,W}$ of $(\what Q,W)$ is given by
$$
\sum_{\al\in\cN^{Q_0}}[\gM(J_{\what Q,W},\al)]_\vir y^\al
%=\sum_{\al\in\cN^{Q_0}}\cL^{\hi_{Q}(\al,\al)}\frac{[R(\Pi_Q,\al)]}{[\GL_\al]}y^\al
=\Exp\left(\frac{\sum_{\al\in\cN^{Q_0}}a_\al(\cL) y^\al}{1-\cL\inv}\right).$$
\end{thr}

Positivity conjecture of Kac \cite{kac_root} states that the polynomials $a_\al$ have non-negative integer coefficients. In view of the above theorem this is equivalent to the statement that the motivic DT invariants of $(\what Q,W)$ are polynomials with non-negative integer coefficients. We generalize this statement in Conjecture \ref{conj:nonneg} for a broad class of quivers with potentials. For quivers with the trivial potential this was conjectured by Kontsevich and Soibelman \cite{kontsevich_cohomological} and proved by Efimov \cite{efimov_cohomological}. In the case of a conifold our conjecture can be verified using the explicit formula \cite[Theorem 2.1]{morrison_motivica}. The Kac positivity conjecture is still open, although some progress was done in \cite{crawley-boevey_absolutely,mozgovoy_motivic}. The above theorem provides the geometric meaning of the polynomials $a_\al$. Earlier this was done only for indivisible $\al\in\cN^{Q_0}$ \cite{crawley-boevey_absolutely}. One can hope that the above theorem together with Conjecture \ref{conj:nonneg} will give a natural way to prove Kac positivity conjecture.

Now let us go back to the crepant resolution $Y=\Hilb^{G}(\cC^3)$ of $\cC^3/G$. The corresponding quiver $Q$ is of affine type in this case and we can easily compute the polynomials $a_\al$. Namely, if \al is a real root then $a_\al(q)=1$ and if \al is an imaginary root then $a_\al(q)=q+l$, where $l$ is the number of isomorphism classes of non-trivial irreducible representations of $G$. Using Theorem \ref{main}, we obtain an explicit formula for the universal motivic DT series.

Following \cite{nagao_counting}, for any stability parameter $\ze\in\cR^{Q_0}$, we can define the moduli spaces $\gM_\ze(J'_{\what Q,W},\al)$ of framed \ze-semistable $J_{\what Q,W}$-modules. The generating function
$$\lZ_\ze=\sum_{\al\in\cN^{Q_0}}[\gM_\ze(J'_{\what Q,W},\al)]_\vir y^\al$$
of their virtual motives can be determined from the universal motivic DT series of $J_{\what Q,W}$ using results from \cite{mozgovoy_wall-crossing} (see Corollary \ref{crl:factorization}).

\begin{thr}
For any generic stability parameter $\ze\in\cR^{Q_0}$ we have
$$\lZ_\ze=\prod_{\ze\cdot\al<0}\lZ_\al$$
where
\begin{equation*}
\lZ_\al(-y_0,y_1,\dots)=\case{
\prod_{j=1}^{\al_0}(1-\cL^{j-\frac{\al_0}2}y^\al)\inv
&\al\in\De_+^\re
\\
\prod_{j=1}^{\al_0}(1-\cL^{j+1-\frac{\al_0}2}y^\al)\inv(1-\cL^{j-\frac{\al_0}2}y^\al)^{-l}
&\al\in\De_+^\im
\\
1&\text{otherwise}}
\end{equation*}
\end{thr}

Applying this result to the stability parameters corresponding to the DT/PT moduli spaces on $Y$ we obtain (see Corollary \ref{crl:PT/DT})

\begin{crl}
We have
\begin{align*}
\lZ_{PT}(Y,-s,Q)
&=\sum_{\be,n}[P_n(Y,\be)]_\vir s^nQ^\be
=\prod_{n\ge1}\prod_{j=1}^n\prod_{\al\in\Deo_+}(1-\cL^{j-\frac n2}s^nQ^\al)\inv,\\
\lZ_{DT}(Y,-s,Q)
&=\sum_{\be,n}[I_n(Y,\be)]_\vir s^nQ^\be\\
&=\lZ_{PT}(Y,-s,Q)\cdot\prod_{n\ge1}\prod_{j=1}^{n}(1-\cL^{j+1-\frac{n}2}s^n)\inv(1-\cL^{j-\frac{n}2}s^n)^{-l}.
\end{align*}
\end{crl}

A closely related result for motivic non-commutative DT invariants is discussed in Remark \ref{rmr:NCDT formula}. The specialization of these results at $\cL^\oh=1$ gives the numerical DT/PT/NCDT invariants of $Y$ which were obtained earlier by Gholampour and Jiang \cite[Theorem 1.2 and Theorem 1.7]{gholampour_counting} using localization technique.  For abelian $G$ numerical NCDT invariants were computed by Young \cite{young_generating} using combinatorial methods and DT/PT/NCDT invariants were computed by Nagao \cite{nagao_derived} using wall-crossing formulas. While writing this paper I was informed by Andrew Morrison that he obtained a similar result for the motivic NCDT invariants in the case of abelian $G$.

The paper is organized as follows. In section \ref{prelim} we recall the notions of motivic rings, \la-rings, moduli spaces of quiver representation, virtual motives, quantum tori, and wall-crossing formulas. In Theorem \ref{aut} we will prove a useful result that allows to compute the motive of the automorphism group of an object in a Krull-Schmidt category.
In Section \ref{sec:perverse} we will recall the technique of Nagao and Nakajima \cite{nagao_counting} that allows to describe the moduli spaces of PT/DT invariants on a small crepant resolution of a singular affine $3$-Calabi-Yau variety in terms of the moduli spaces of representations of its non-commutative crepant resolution. Similar description for the case of McKay quivers can be found in \cite{gholampour_counting}. We will also discuss the correspondence between the topological invariants of coherent sheaves on the crepant resolution and dimension vectors of representations of the non-commutative crepant resolution. In section \ref{sec:mckay} we will see how McKay correspondence provides an example of a general framework discussed in Section \ref{sec:perverse}. In Section \ref{sec:loop} we will compute universal motivic DT series of the quiver with potential $(\what Q,W)$ for an arbitrary quiver $Q$. Here we also formulate the positivity conjecture for the motivic DT invariants of quivers with potentials. In Section \ref{sec:applications} we compute motivic DT/PT invariants in the McKay situation.

I would like to thank Tamas Hausel, Andrew Hubery, Kentaro Nagao, Markus Reineke, and Bal\'azs Szendr\H oi for many helpful discussions. The author's research was supported by EPSRC grant EP/G027110/1.

%% file: prelim.tex
\section{Preliminaries}
\label{prelim}
\subsection{Ring of motives}
Let $K_0(\CM_\cC)$ be the Grothendieck ring of the category of Chow motives over \cC with rational coefficients. It is known that $K_0(\CM_\cC)$ is a (special) \la-ring \cite{getzler_mixed,heinloth_note} with \si-operations defined by $\si_n([X])=[X^n/S_n]$ for any smooth projective variety $X$. Let $\cL=[\cA^1]\in K_0(\CM_\cC)$ be the Lefschetz motive. The ring $\lM_\cC=K_0(\CM_\cC)[\cL^{-\oh}]$ also has a \la-ring structure with \si-operations extended by $\si_n(\cL^\oh)=\cL^{\frac n2}$.
We define the ring $\lvc$ to be the dimensional completion of $\lM_\cC$ with respect to $\cL\inv$ (see \cite{behrend_motivica,mozgovoy_wall-crossing}). The \la-ring structure on $\lM_\cC$ can be extended to $\lvc$. The elements $1-\cL^n$ and $[\GL_n]$ are invertible in~$\lvc$. The last statement follows from the fact that
\begin{equation}
[\GL_n]=\prod_{k=0}^{n-1}(\cL^n-\cL^k)
=\cL^{n^2}\prod_{k=1}^n(1-\cL^{-k})=\cL^{n^2}(\cL\inv)_n,
\label{eq:GL_n}
\end{equation}
where $(q)_n=(q;q)_n=\prod_{k=1}^n(1-q^k)$ are the $q$-Pochhammer symbols.

The map sending a smooth projective variety $X$ to its $E$-polynomial
$$E(X,u,v)=\sum_{p,q\ge0}(-1)^{p+q}\dim H^{p,q}(X,\cC)u^pv^q$$
can be extended to the \la-ring homomorphism
$$E:\lvc\to\cQ[u,v]\pser{(uv)^{-\oh}}$$
with $E(\cL^\oh)=(uv)^\oh$. It can be specialized to the Poincar\'e polynomial $$P:\lvc\to\cQ\lser{y^{-1}},\qquad P(X,y)=E(X,y,y).$$
There exists also the Euler number specialization $e:\lM_\cC\to\cQ$, $e(X)=E(X,1,1)$, which is a \la-ring homomorphism. Note that we can not extend $e$ to $\lvc$, as for example the image of $\sum_{n\ge0}\cL^{-n}$ would not converge.

\subsection{Motive of the group of automorphisms}
An additive category is called a Krull-Schmidt category if any of its objects can be decomposed into a finite direct sum of indecomposable objects and endomorphism rings of indecomposable objects are local. By the Krull-Schmidt theorem a decomposition of an object from such category into a direct sum of indecomposable objects is unique up to a permutation of direct summands.
Let $k$ be a field and let \lA be an additive $k$-linear category with finite-dimensional $\Hom$-spaces and with splitting idempotents (if \lA is abelian then all its idempotents automatically split). 
Then \lA is a Krull-Schmidt category (it is called a Krull-Schmidt $k$-category).

\begin{thr}
\label{aut}
Let \lA be a Krull-Schmidt \cC-category. Given an object $X\in\lA$, let $X=\oplus_{i\in I}X_i^{n_i}$ be its decomposition into the sum of indecomposable objects. Then the motive of the group $\Aut(X)$ is
$$[\Aut(X)]=[\End(X)]\cdot{\prod_{i\in I}(\cL\inv)_{n_i}}.$$
\end{thr}
\begin{proof}
Let $\lR_\lA$ be the radical of the category \lA \cite[Section 3.2]{gabriel_representations}. If $X,Y\in\lA$ are non-isomorphic indecomposable objects then $\lR_\lA(X,Y)=\lA(X,Y)$. If $X\in\lA$ is indecomposable then $\lR_\lA(X,X)\iso J(\lA(X,X))$, the Jacobson radical of the local ring $\End(X)=\lA(X,X)$.
Using the decomposition $X=\oplus_{i\in I}X_i^{n_i}$ we can write
$$\lA(X,X)/\lR_\lA(X,X)=\prod_{i\in I}\End(\cC^{n_i}).$$
An element in $\lA(X,X)$ is invertible if and only if it is invertible in $(\lA/\lR_\lA)(X,X)$. This implies
$$[\Aut(X)]
=[\lR_\lA(X,X)]\prod_{i\in I}[\GL_{n_i}]
=[\End(X)]\prod_{i\in I}\frac{[\GL_{n_i}]}{\cL^{n_i^2}}.$$
We use now the formula $[\GL_n]=\cL^{n^2}(\cL\inv)_n$ from \eqref{eq:GL_n}.
\end{proof}

\subsection{\texorpdfstring{\la}{lambda}-rings and power structures}
For basic definitions and constructions related to \la-rings see \eg \cite{getzler_mixed,mozgovoy_computational}. For simplicity we will assume that all our \la-rings are algebras over \cQ and therefore the \la-ring structure is uniquely determined by Adams operations.
Given a \la-ring $R$, we can endow the ring $\hat R=R\pser{y_1,\dots,y_m}$ with a \la-ring structure by defining the Adams operations
$$\psi_n(ry^\al)=\psi_n(r)y^{n\al},\qquad r\in R,\al\in\cN^{m}.$$

%By http://mathoverflow.net/questions/13486/is-every-adams-ring-morphism-a-lambda-ring-morphism
%there can exist different \la-ring structures inducing the same Adams operations 

Let $\hat R_+\sb\hat R$ be an ideal generated by $y_1,\dots,y_m$. 
Define a map $\Exp:\hat R_+\to1+\hat R_+$ by \cite{getzler_mixed}
$$\Exp(f)=\sum_{n\ge0}\si_n(f)=\exp\bigg(\sum_{n\ge1}\frac1n\psi_n(f)\bigg).$$
It is proved in \cite[Prop.~2.2]{getzler_mixed} (see also \cite[Cor.~21]{mozgovoy_computational}) that $\Exp$ has inverse $\Log:1+\hat R_+\to \hat R_+$
$$\Log(f)=\sum_{n\ge1}\frac{\mu(n)}{n}\psi_n\log(f),$$
where $\mu$ is the M\"obius function.

For example, let $R=\cQ\lser q$ be endowed with a \la-ring structure by $\psi_n(f(q))=f(q^n)$. Then, applying the q-binomial theorem (see \eg Heine \cite[Eq.74]{heine_untersuchungen}), we obtain in the ring 
$R\pser x$
\begin{equation}
\sum_{n\ge1}\frac{x^n}{(q)_n}=\prod_{k\ge0}\frac{1}{1-xq^k}=\prod_{k\ge0}\Exp(xq^k)=\Exp\Big(\frac{x}{1-q}\Big),
\label{eq:heine}
\end{equation}
where, as before, $(q)_n=(q;q)_n=\prod_{k=0}^{n-1}(1-q^k)$.

Following \cite{mozgovoy_computational}, we define a power structure map
$$\Pow:(1+\hat R_+)\xx \hat R\to1+\hat R_+,\qquad (f,g)\mto \Exp(g\Log(f)).$$
This map has an interesting geometric description if $R=\wtl\lM_\cC$ is a ring of motives~\cite{gusein-zade_powerb}.
Let 
$$f=\sum_{\al\in \cN^m}f_\al=1+\sum_{\al>0}[A_\al]y^\al,$$
where $A_\al$ are algebraic varieties, and let $X$ be an algebraic variety. Let $A=\coprod_{\al>0}A_\al$ and let $\deg:A\to \cN^m$ be given by $A_\al\ni x\mto\al$.
According to \cite{gusein-zade_powerb,gusein-zade_powera}, the coefficient of $y^\be$ in $\Pow(f,[X])$ is given by the motive of the configuration space of pairs $(K,\vi)$, where 
\begin{enumerate}
	\item $K$ is a finite subset of $X$,
	\item $\vi:K\to A$ is a map such that $\sum_{x\in K}\deg\vi(x)=\be$.
\end{enumerate}
The geometric description of the above space of pairs is the following.
Define the type of a pair $(K,\vi)$ to be the map $k:\cN^m\to\cN$ given by
$$k(\al)=\#\sets{x\in K}{\deg\vi(x)=\al}.$$
The pair $(K,\vi)$ satisfies the condition $\sum_{x\in K}\deg\vi(x)=\be$ if and only if 
\begin{equation}
\sum_{\al\in \cN^m}k(\al)\al=\be.
\label{eq:geom deg}
\end{equation}
There is just a finite number of maps $k:\cN^m\to\cN$ satisfying this condition. The space of pairs $(K,\vi)$ of type $k$ can be parametrized by \cite{gusein-zade_powerb,gusein-zade_powera}
\begin{equation}
\Big(F_{\n k}X\xx\prod_{\al\in \cN^m}A_\al^{k(\al)}\Big)/\prod_{\al\in \cN^m} S_{k(\al)},
\label{eq:1}
\end{equation}
where $\n k=\sum_{\al\in \cN^m}k(\al)$, the configuration space $F_nX$ is given by
$$F_nX=\sets{(x_1,\dots,x_n)\in X^n}{x_i\ne x_j\text{ for }i\ne j},$$
and the product of symmetric groups $\prod_{\al\in \cN^m} S_{k(\al)}$ acts on both factors of \eqref{eq:1} in the obvious way. 
Applying \eqref{eq:geom deg} and \eqref{eq:1} we obtain
\begin{equation}
\Pow(f,[X])=
\sum_{k:\cN^m\to\cN}
\bigg[\Big(F_{\n k}X\xx\prod_{\al\in \cN^m}A_\al^{k(\al)}\Big)/\prod_{\al\in \cN^m} S_{k(\al)}\bigg]y^{\sum k(\al)\al}.
\label{eq:pow geom}
\end{equation}

Let us give a different parametrization of pairs $(K,\vi)$. With any pair $(K,\vi)$ we can associate a map $\psi:X\to \cN^m$ with finite support (\ie $\psi\inv(\cN^m\ms\set0)$ is finite) given by
$$\psi(x)=\case{\deg\vi(x)&x\in K,\\0&x\not\in K.}$$
%Note that if $(K,\vi)$ is of type $k$ then
%$$\sum_{x\in K}\deg\vi(x)=\sum_{\al\in \cN^m}k(\al)\al=\sum_{x\in X}\psi(x).$$
The pairs $(K,\vi)$ corresponding to the given map $\psi:X\to \cN^m$ are parametrized by
$\prod_{x\in \cN^m}A_{\psi(x)}$.
This means that we can write the coefficient of $y^\be$ in $\Pow(f,[X])$ as
$$\sum_{\over{\psi:X\to \cN^m}{\sum_{}\psi(x)=\be}}\prod_{x\in X}[A_{\psi(x)}]$$
which should be interpreted using the parametrization in \eqref{eq:1}. We can write now
% the whole expression $\Pow(f,[X])$ as 
\begin{equation}
\Pow(f,[X])=\sum_{\psi:X\to \cN^m}\prod_{x\in X}f_{\psi(x)},
\label{power}
\end{equation}
where the sum runs over all maps $\psi:X\to \cN^m$ with finite support. This formula can be used for arbitrary $f_\al=[A_\al]y^\al$ (with $[A_\al]$ an arbitrary motive and not necessarily a motive of an algebraic variety). For this we just have to interpret \eqref{eq:1} appropriately.

\subsection{Quivers, moduli stacks, virtual motives}
\subsubsection{Quivers with potentials and their representations}
Let $(Q,W)$ be a quiver with a potential. Let $J=J_{Q,W}=\cC Q/(\dd W)$ be the corresponding Jacobian algebra.

Given a $Q$-representation $M$, we define its dimension vector
$$\udim M=(\dim M_i)_{i\in Q_0}\in\cN^{Q_0}.$$
For any $\al\in\cN^{Q_0}$, we define the space of $Q$-representations with dimension vector \al to be
$$R(Q,\al)=\bop_{(a:i\to j)\in Q_1}\Hom(\cC^{\al_i},\cC^{\al_j}).$$
Let $R(J,\al)$ be the subset of $R(Q,\al)$ consisting of $Q$-representations that satisfy the relations of $J$.
The group $G_\al=\prod_{i\in Q_0}\GL_{\al_i}(\cC)$ acts on $R(Q,\al)$ and $R(J,\al)$.

Define the Euler-Ringel form $\hi_Q$ by
$$\hi_Q(\al,\be)=\sum_{i\in Q_0}\al_i\be_i-\sum_{(a:i\to j)\in Q_1}\al_i\be_j,\qquad \al,\be\in\cZ^{Q_0}.$$
Then $\dim R(Q,\al)-\dim G_\al=-\hi_Q(\al,\al)$. Define the skew-symmetric form
$$\ang{\al,\be}=\hi_Q(\al,\be)-\hi_Q(\be,\al),\qquad \al,\be\in\cZ^{Q_0}.$$

For any $Q$-representation $M$, let $w(M)\in\cC$ be obtained by taking the trace of the linear map on $M$ associated to $W$. This defines a $G_\al$-invariant map
$w_\al:R(Q,\al)\to\cC$. It is known that the set of critical points of $w_\al$ coincides with $R(J,\al)$.

Given $\ze\in\cR^{Q_0}$, called a stability parameter, we define the slope function $\mu_\ze:\cN^{Q_0}\ms\set0\to\cR$ by the rule
$$\mu_\ze(\al)=\frac{\ze\cdot\al}{\n\al},$$
where $\n\al=\sum_{i\in Q_0}\al_i$. For any nonzero $Q$-representation $M$, we define $\mu_\ze(M)=\mu_\ze(\udim M)$. A $Q$-representation $M$ is called \ze-semistable (resp.\ \ze-stable) if for any proper nonzero submodule $N\sb M$ we have $\mu_\ze(N)\le\mu_\ze(M)$ (resp.\ $\mu_\ze(N)<\mu_\ze(M)$). In the same way we define the notion of \ze-(semi)stability for $J$-modules.

\subsubsection{Moduli stacks and their motives}
We define the stacks of $Q$-representations and $J$-modules with dimension vector \al to be
\begin{equation}
\gM(Q,\al)=[R(Q,\al)/G_\al],\qquad \gM(J,\al)=[R(J,\al)/G_\al].
\label{eq:p01}
\end{equation}

Let $R_\ze(Q,\al)$ (resp.\ $R_\ze(J,\al)$) be the open subset of $R(Q,\al)$ (resp.\ $R(J,\al)$) consisting of \ze-semistable $Q$-representations (resp.\ $J$-modules). We define the moduli stacks
\begin{equation}
\gM_\ze(Q,\al)=[R_\ze(Q,\al)/G_\al],\qquad \gM_\ze(J,\al)=[R_\ze(J,\al)/G_\al].
\label{eq:p02}
\end{equation}

\begin{rmr}
\label{rem:cut}
For technical reasons we always assume that there exists a cut of $(Q,W)$. This is a subset $I\sb Q_1$ such that $W$ is homogeneous of degree $1$ with respect to the weight function $\wt:Q_1\to\cN$ defined by
$$\wt(a)=\case{1&a\in I,\\0&a\in Q_1\ms I.}$$
Note that such weight function defines an action of $\cC^*$ on $R(Q,\al)$, $\al\in\cN^{Q_0}$. We have $w_\al(tM)=tw_\al(M)$ for any $t\in\cC^*$, $M\in R(Q,\al)$.
\end{rmr}

The map $w_\al:R(Q,\al)\to\cC$ restricts to the map $w_{\ze,\al}:R_\ze(Q,\al)\to\cC$ and its critical locus is $R_\ze(J,\al)$. In order to define the virtual motive \cite[Def.1.13]{behrend_motivic} of $\gM_\ze(J,\al)$ (or $R_\ze(J,\al)$) one uses the motivic vanishing cycle of $w_{\ze,\al}$. According to \cite[Prop.1.10]{behrend_motivic} the vanishing cycle is greatly simplified if there exists an appropriate torus action on $R_\ze(Q,\al)$.
Such action exists in our situation (see Remark \ref{rem:cut}) and therefore, following \cite{mozgovoy_motivica,nagao_wall-crossing}, we define the virtual motive
\begin{equation}
[\gM_\ze(J,\al)]_\vir=(-\cL^\oh)^{\hi(\al,\al)}
\frac{[w_{\ze,\al}\inv(0)]-[w_{\ze,\al}\inv(1)]}{[G_\al]}.
\label{eq:p03}
\end{equation}
Taking the trivial stability $\ze=0$, we get the virtual motive 
\begin{equation}
[\gM(J,\al)]_\vir=(-\cL^\oh)^{\hi(\al,\al)}
\frac{[w_{\al}\inv(0)]-[w_{\al}\inv(1)]}{[G_\al]}.
\label{eq:p04}
\end{equation}

There is an easy way to compute $[\gM(J,\al)]_\vir$.
Let $I\sb Q_1$ be the cut of $(Q,W)$ and let $Q_I$ be the new quiver defined by $Q_I=(Q_0,Q_1\ms I)$. Define the algebra
$$J_{W,I}=\cC Q_I/(\dd_aW,a\in I).$$
The following result was proved in \cite{morrison_motivic,nagao_wall-crossing}.

\begin{prp}[First dimensional reduction]
\label{prp:reduction}
For any $\al\in\cN^{Q_0}$ we have
$$[w_\al\inv(0)]-[w_\al\inv(1)]=\cL^{d_I(\al)}[R(J_{W,I},\al)],$$
where $d_I(\al)=\sum_{(a:i\to j)\in I}\al_i\al_j$. In particular
$$[\gM(J,\al)]_\vir=(-\cL^\oh)^{\hi(\al,\al)+2d_I(\al)}\frac{[R(J_{W,I},\al)]}{[G_\al]}.$$
\end{prp}

\subsubsection{Quantum torus and factorization formula}
Define the motivic quantum torus $\qt=\qt_Q$ to be the algebra given by the vector space
$$\lvc\pser{y_i,i\in Q_0}$$
with multiplication
$$y^\al\circ y^\be=(-\cL^\oh)^{\ang{\al,\be}}y^{\al+\be}.$$
We organize the virtual motives defined earlier in generating functions in \qt. Namely, we define
\begin{equation}
A_U=\sum_{\al\in\cN^{Q_0}}[\gM(J,\al)]_\vir y^\al,\qquad A_{\ze,\mu}
=\sum_{\over{\al\in\cN^{Q_0}}{\mu_\ze(\al)=\mu}}[\gM_\ze(J,\al)]_\vir y^\al
\label{eq:p05}
\end{equation}
for $\mu\in\cR$. The following result was proved in \cite{mozgovoy_motivica,nagao_wall-crossing}. A stronger result without the assumption of the existence of a cut was proved in \cite{kontsevich_stability}.

\begin{prp}
\label{factor}
For any stability parameter \ze we have
$$A_U=\prod^{\leftarrow}_{\mu\in\cR} A_{\ze,\mu},$$
where the product is taken in the decreasing order of $\mu\in\cR$.
\end{prp}

If $Q$ is a symmetric quiver (\ie the number of arrows from $i$ to $j$ equals the number of arrows from $j$ to $i$ for any $i,j\in Q_0$) then the quantum torus \qt is commutative. In this case we define motivic Donaldson-Thomas invariants $\Om_\al\in\lvc$, $\al\in\cN^{Q_0}$, by the formula
\begin{equation}
A_U=\Exp\left(\frac{\sum_\al\Om_\al y^\al}{1-\cL\inv}\right).
\label{eq:DT}
\end{equation}

\subsubsection{Framed quiver representations}
\label{sec:framed}
Let $(Q,W)$ be as before and let $w\in\cN^{Q_0}$. We define a new quiver $Q'$ by adding one new vertex $\infty$ to $Q$ and adding $w_i$ arrows from $\infty$ to $i$ for every $i\in Q_0$. Considering $W$ as a potential in $Q'$ we can define the Jacobian algebra $J'=J_{Q',W}$.

Given a stability parameter $\ze\in\cR^{Q_0}$ and a dimension vector $\al\in\cN^{Q_0}$, we consider the moduli stacks $\gM_{\ze'}(Q',\al')$ and $\gM_{\ze'}(J',\al')$, where $\al'=(\al,1)$ and $\ze'=(\ze,\ze_\infty)$ with $\ze_\infty=-\ze\cdot\al$ (this condition means that $\ze'\cdot\al'=0$).

\begin{rmr}
The stack $\gM(Q',0')$ consists of one $1$-dimensional representation concentrated at vertex $\infty\in Q'_0$. We have
$$\lZ_\infty=[\gM(Q',0')]_\vir=\frac{(-\cL^\oh)^{\hi_{Q'}((0,1),(0,1))}}{[\GL_1]}=\frac{-\cL^\oh}{\cL-1}.$$
We normalize virtual motives with respect to this one and define
$$[\gM(Q',\al)]_\vir=\lZ_\infty\inv\cdot [\gM(Q',\al')]_\vir,$$
which is the virtual motive of the stack $\gM(Q',\al)=[R(Q',\al')/\GL_\al]$ (note that here we take the quotient stack with respect to $\GL_\al$ and not with respect to $\GL_{\al'}=\GL_\al\xx\cC^*$). In the same way we define
$$[\gM_\ze(Q',\al)]_\vir=\lZ_\infty\inv\cdot [\gM_{\ze'}(Q',\al')]_\vir,\qquad
[\gM_\ze(J',\al)]_\vir=\lZ_\infty\inv\cdot [\gM_{\ze'}(J',\al')]_\vir.$$
\end{rmr}

We define the generating series of virtual motives
$$\lZ_\ze=\sum_{\al\in\cN^{Q_0}}[\gM_\ze(J',\al)]_\vir y^\al\in\qt.$$

\begin{rmr}
\label{rmr:NCDT}
Let $\ze\in\cR^{Q_0}$ be such that $\ze_i=-1$, $i\in Q_0$. The moduli stack $\gM_\ze(Q',\al)$ consists of $Q'$-representations $M$ generated by $M_\infty$. The virtual motives $[\gM_\ze(J',\al)]_\vir$ are called the non-commutative Donaldson-Thomas invariants of $J'$ and are denoted by $[\gM_{NCDT}(J',\al)]_\vir$. The corresponding generating function is denoted by $\lZ_{NCDT}$ \cite{szendroi_non-commutative}.
\end{rmr}

The following result was proved in \cite[Cor.~4.17]{mozgovoy_motivica}.

\begin{prp}
For any stability parameter $\ze\in\cR^{Q_0}$ we have
$$\lZ_\ze=S_w\bigg(\prod_{\mu\le0}^\leftarrow A_{\ze,\mu}\bigg)\circ 
S_{-w}\bigg(\prod_{\mu<0}^\leftarrow A_{\ze,\mu}\bigg)\inv,$$
where, for any $v\in\cZ^{Q_0}$, we define $S_v:\qt\to\qt$, $y^\al\mto (-\cL^\oh)^{v\cdot\al}y^\al$.
\end{prp}

Assume that $Q$ is symmetric and the Donaldson-Thomas invariants $\Om_\al\in\lvc$ are defined as in \eqref{eq:DT}.
%$$A_U=\Exp\left(\frac{\sum_\al\Om_\al y^\al}{\cL-1}\right).$$
Then, according to Proposition \ref{factor}
$$\prod_{\mu\le0}^\leftarrow A_{\ze,\mu}=
\Exp\left(\frac{\sum_{\ze\cdot\al\le0}\Om_\al y^\al}{1-\cL\inv}\right),\qquad
\prod_{\mu<0}^\leftarrow A_{\ze,\mu}=
\Exp\left(\frac{\sum_{\ze\cdot\al<0}\Om_\al y^\al}{1-\cL\inv}\right)$$
and this means that we can compute $\lZ_\ze$ for any stability parameter $\ze\in\cR^{Q_0}$ if we know $A_U$.

\begin{crl}
\label{crl:Z factoriz}
Assume that $Q$ is symmetric and $\ze\in\cR^{Q_0}$ is such that $\ze\cdot\al\ne0$ whenever $\Om_\al\ne0$ (we will say that \ze is generic in this case). Then
$$\lZ_\ze=S_{-w}\Exp\Bigg(\sum_{\ze\cdot\al<0}\frac{\cL^{w\cdot\al}-1}{1-\cL\inv}\Om_\al y^\al\Bigg).$$
%$$\lZ_\ze=\prod_{\ze\cdot\al<0}\lZ_\al,$$
%where
%$$\lZ_\al=S_w\Exp\left(\frac{\Om_\al y^\al}{\cL-1}\right)/
%S_{-w}\Exp\left(\frac{\Om_\al y^\al}{\cL-1}\right)
%=S_{-w}\Exp\left(\frac{\cL^{w\cdot\al}-1}{\cL-1}\Om_\al y^\al\right).$$
\end{crl}
 
\begin{rmr}
Using the last corollary we can compute the Euler number specialization
$$\ub\lZ_\ze=\ub S_w\Exp\bigg(\sum_{\ze\cdot\al<0}(w\cdot\al)\ub\Om_\al y^\al\bigg),$$
%$$\ub\lZ_\al=\ub S_w\Exp((w\cdot\al)\ub\Om_\al y^\al),$$
where $\ub S_w$ is given by $y^\al\mto(-1)^{w\cdot\al}y^\al$.
\end{rmr}

%% file: s01.tex
\section{Perverse coherent sheaves and DT/PT invariants}
\label{sec:perverse}
The goal of this section is to introduce the technique due to Nagao and Nakajima \cite{nagao_counting} that in some situations allows to interpret Pandharipande-Thomas moduli spaces of stable pairs (as well as Donaldson-Thomas moduli spaces) on a small crepant resolution of a singular affine $3$-Calabi-Yau variety in terms of the moduli spaces of representations of its non-commutative crepant resolution.

\subsection{Riemann-Roch theorem}
Let $Y$ be a $3$-Calabi-Yau manifold. Denote by $\Coh_{c,\le1}$ the category of coherent sheaves over $Y$ having compact support of dimension $\le1$. For any $F\in\Coh_{c,\le1}$ we have $\ch_0F=0$, $\ch_1F=0$. The pair
$$(\ch_2F,\ch_3F)\in H^4_c(Y,\cZ)\oplus H^6_c(Y,\cZ)$$
is given by $(\be,n)\in H_2(Y,\cZ)\oplus\cZ$, where we identify $H^4_c(Y,\cZ)\iso H_2(Y,\cZ)$ and $H^6_c(Y,\cZ)\iso H_0(Y,\cZ)\iso\cZ$ using Poincar\'e duality, and where $\be$ is the class of the support of $F$ and $n=\hi(F)$. The last equation is a consequence of the following result.

\begin{lmm}[\cf {\cite[Section 3.2]{toda_stabilitya}}]
\label{lmm:RR}
Assume that $Y$ is a $3$-Calabi-Yau manifold. Then for any $E\in D^b(\Coh Y)$, $F\in D^b(\Coh_{c,\le1}Y)$ we have
$$\hi(E,F)=\sum_{i\in\cZ}(-1)^i\dim\Hom(E,F[i])=\ch_0E\ch_3F-\ch_1E\ch_2F.$$
\end{lmm}
\begin{proof}
It follows from the Riemann-Roch theorem that
$$\hi(E,F)=\ch_0E\ch_3F-\ch_1E\ch_2F+\ch_0E\ch_2F\td_1Y.$$
By the Calabi-Yau condition $\td_1Y=0$.
\end{proof}

\subsection{Perverse coherent sheaves}
\label{sec:perv}
Let $f:Y\to X$ be a projective morphism between
algebraic varieties such that the fibers of $f$ have dimension $\le1$, $Rf_*\lO_Y=\lO_X$, and $X=\Spec R$ is affine. 

Following \cite{bridgeland_flops,vandenbergh_three-dimensional}, we define $\Per(Y/X)={}\inv\Per(Y/X)$ to be the subcategory of $D^b(Y)=D^b(\Coh Y)$ consisting of objects $E$ such that
\begin{enumerate}
	\item $H^i(E)=0$ for $i\not\in\set{0,-1}$,
	\item $R^1f_*H^0(E)=0$, $\Hom(H^0(E),F)=0$ if $Rf_*F=0$,
	\item $R^0f_*H^{-1}(E)=0$.
\end{enumerate}
It is shown in \cite{bridgeland_flops,vandenbergh_three-dimensional} that $\Per(Y/X)$ is the heart of some $t$-structure on $D^b(Y)$.

Let $\lP\in\Per(Y/X)$ be a projective generator and let $J=\End_Y(\lP)\op$.
We denote by $\mod J$ the category of finitely-generated left $J$-modules, or equivalently, the category of $J$-modules finitely generated over $R$.
Then the functors
\begin{align*}
\Phi=&\RHom(\lP,-):D^b(\Coh Y)\to D^b(\mod J),\\
\Psi=&\lP\lts-:D^b(\mod J)\to D^b(\Coh Y)	
\end{align*}
are inverse equivalences which restrict to the equivalences between $\Per(Y/X)$ and $\mod J$ \cite[Corollary 3.2.8]{vandenbergh_three-dimensional}.

Let $\Coh_cY$ denote the category of coherent sheaves over $Y$ having compact support. The image of this support in $X$ is $0$-dimensional, as $X$ is affine. Therefore the support of any sheaf in $\Coh_c Y$ has dimension $\le1$ and $\Coh_cY=\Coh_{c,\le1}Y$. Let $\mod_c J$ denote the category of modules over $J$ having compact (\ie $0$-dimensional) support as sheaves over $X$. Equivalently, $\mod_c J$ consists of finite-dimensional $J$-modules. Let $D^b_c(Y)=D^b_c(\Coh Y)$ (resp.\ $D^b_c(\mod J)$) denote the category of objects with cohomologies in $\Coh_c Y$ (resp.\ $\mod_cJ$). Let $\Per_c(Y/X)$ be the intersection of $\Per(Y/X)$ with $D^b_c(Y)$. The functors $\Phi,\Psi$ restrict to the equivalences between $D^b_c(Y)$ and $D^b_c(\mod J)$ and between $\Per_c(Y/X)$ and $\mod_cJ$.

\subsection{Category of triples}
\label{sec:triples}
From now on we will assume that $\lP=\oplus_{i\in I}\lP_i$ and there is a distinguished element $0\in I$ such that $\lP_0=\lO_Y$  and $(c_1(\lP_i))_{i\in I\ms\set0}$ form a basis of $H^2(Y,\cZ)$. Let $I_*=I\ms\set0$ and $\lP_*=\sum_{i\in I_*}\lP_i$.
We can decompose $\ub F=\Phi(F)=\oplus_{i\in I}\ub F_i$, where $\ub F_i=\RHom_Y(\lP_i,F)$.
If $F\in\Per(Y/X)$ then $\ub F_i=\Hom_Y(\lP_i,F)$.

Let $(e_i)_{i\in I}$ be idempotents in $J$ corresponding to the decomposition $\lP=\oplus_{i\in I}\lP_i$.
%Let $e_*=\sum_{i\in I_*}e_i$.
Any $J$-module $M$ can be decomposed as $M=\oplus_{i\in I}M_i$,
%=M_0\oplus M_*$,
where $M_i=e_iM$. 
Define the abelian category $\mod J'$ to be the category consisting of triples $(M,M_\infty,s)$, where $M\in\mod J$, $M_\infty$ is a vector space and $s:M_\infty\to M_0$ is a linear map. We can define the corresponding algebra $J'$ to be generated by $J$ and two elements $e_\infty,s$ subject to the relations
$$e_\infty 1_J=1_Je_\infty=0,\quad e_\infty^2=e_\infty,\quad se_\infty=s,\quad e_0s=s.$$ 
We denote by $\mod_cJ'$ the category of finite-dimensional $J'$-modules.

\begin{rmr}
Given a morphism $s:V\ts\lO_Y\to F$ in $D^b(Y)$ with $F\in\Per_c(Y/X)$ and $V$ a vector space, we can associate with it a triple $(\ub F,V,\ub s)\in\mod J'$ using the fact that
$$s\in\Hom_Y(V\ts\lO_Y,F)\iso\Hom(V,\Hom_Y(\lO_Y,F))\iso\Hom(V,\ub F_0).$$
\end{rmr}

\subsection{\texorpdfstring{$3$}{3}-Calabi-Yau case}
Assume that $Y$ is a $3$-Calabi-Yau manifold.

\begin{lmm}
Given $F\in\Per_c(Y/X)$ with $(\ch_2F,\ch_3F)=(\be,n)$, we have
$$\dim\ub F_i=n\rk\lP_i-c_1(\lP_i)\cdot\be,\qquad i\in I$$
and in particular $\dim\ub F_0=n$.
\end{lmm}
\begin{proof}
We have $\ub F_i=\RHom(\lP_i,F)=\Hom(\lP_i,F)$. Therefore, applying Lemma \ref{lmm:RR}, we have
$$\dim\ub F_i=\hi(\lP_i,F)=\rk\lP_i\cdot\hi(F)-c_1(\lP_i)\cdot\ch_2F$$
and in particular $\dim\ub F_0=\hi(F)=n$.
\end{proof}

In view of this result
(and our assumption that $(c_1(\lP_i))_{i\in I_*}$ forms a basis of $H^2(Y,\cZ)$) 
we will identify $H_2(Y,\cZ)\oplus \cZ$ with $\cZ^I$ by sending $(\be,n)$ to $\al\in\cZ^I$ given by
\begin{equation}
\al_i=n\rk\lP_i-c_1(\lP_i)\cdot\be,\qquad i\in I.
\label{eq:n,be}
\end{equation}

\subsection{Stability and DT/PT invariants}
\label{DT/PT as reps}
Let as before $Y$ be a $3$-Calabi-Yau manifold.
For any $M\in\mod_c J$ we define $\lb\dim M=(\dim M_i)_{i\in I}\in\cN^I$.
Let $\ze\in\cR^I$. We say that an object $(M,M_\infty,s)\in\mod_cJ'$ with $\dim M_\infty=1$ is $\ze$-stable if it is stable with respect to $(\ze,\ze_\infty)$, where $\ze_\infty=-\ze\cdot\lb\dim M$. This means that for any proper nonzero subobject $(N,N_\infty,t)$ of $(M,M_\infty,s)$ we have
$$
%\frac{
\ze\cdot\lb\dim N+\ze_\infty\dim N_\infty
%}{\dim N+\dim N_\infty}
<0=
%\frac{
\ze\cdot\lb\dim M+\ze_\infty\dim M_\infty
%}{\dim M+\dim M_\infty}
.$$

Let $\ze^\pm=(\ze^\pm_i)_{i\in I}$ be defined by
\begin{equation}
\ze^\pm_0=-\rk\lP_*\pm\eps,\quad \ze^\pm_i=1,\ i\in I_*
\label{eq:ksi}
\end{equation}
for sufficiently small $\eps>0$.

\begin{rmr}
\label{zeta-explanation}
Consider an $f$-ample divisor $\om=c_1(\lP)$. For any $F\in\Per_c(Y/X)$, we have
$\om\cdot\ch_2F=-\ze\cdot\lb\dim\ub F$, where $\ze\in\cZ^I$ is defined by $\ze_0=-\rk\lP_*$ and $\ze_i=1$, $i\in I_*$ (\cf \cite[Prop.~3.5]{gholampour_counting} in the McKay situation). This is the reason for the above choice of $\ze^\pm$.
\end{rmr}

The following result is proved in \cite{nagao_counting}.
% (see also \cite{bridgeland_flops}).

\begin{prp}
\label{NN equiv}
Let $s:\lO_Y\to F$ be some morphism in $D^b(Y)$. Then
\begin{enumerate}
	\item $F\in\Per_c(Y/X)$ and the corresponding $J'$-module $(\ub F,\cC,\ub s)$ is $\ze^-$-stable if and only if $s:\lO_Y\to F$ is a DT-morphism (\ie $F\in\Coh_cY$ and $\coker s=0$).
	\item $F\in\Per_c(Y/X)$ and the corresponding $J'$-module $(\ub F,\cC,\ub s)$ is $\ze^+$-stable if and only if $s:\lO_Y\to F$ is a PT-morphism (\ie $F\in\Coh_cY$ is pure of dimension $1$ and $\coker s$ has $0$-dimensional support).
\end{enumerate}
\end{prp}

Given a pair $(\be,n)\in H_2(Y,\cZ)\oplus\cZ$, let $I_n(Y,\be)$ be the moduli stack of $1$-dimensional subschemes $Z\sb Y$ such that $\hi(\lO_Z)=n$ and $[Z]=\be$. Equivalently, $I_n(Y,\be)$ is the moduli stack of DT-morphisms $f:\lO_X\to F$ such that $\hi(F)=n$, $\ch_2 F=\be$. It follows from Proposition \ref{NN equiv} that this moduli space can be identified with the moduli stack of $\ze^-$-stable $J'$-modules having dimension vector \al given by~\eqref{eq:n,be}.
%$$(n,(n\rk\lP_i-c_1(\lP_i)\cdot\be)_{i\in I_*}).$$

Let $P_n(Y,\be)$ be the moduli stack of PT-morphisms $s:\lO_Y\to F$ with $\hi(F)=n$, $\ch_2(F)=\be$. It follows from Proposition \ref{NN equiv} that this moduli stack can be identified with the moduli stack of $\ze^+$-stable $J'$-modules having dimension vector \al given by~\eqref{eq:n,be}.

We define the series of motivic geometric Donaldson-Thomas invariants by
\begin{equation}
\lZ_{DT}=\sum_{(\be,n)}[I_n(Y,\be)]_\vir s^nQ^\be=\sum_{\al\in\cN^{I}}[\gM_{\ze^-}(J',\al)]_\vir y^\al
\label{eq:gDT}
\end{equation}
and the series of motivic Pandharipande-Thomas invariants by
\begin{equation}
\lZ_{PT}=\sum_{(\be,n)}[P_n(Y,\be)]_\vir s^nQ^\be=\sum_{\al\in\cN^{I}}[\gM_{\ze^+}(J',\al)]_\vir y^\al,
\label{eq:PT}
\end{equation}
where we identify $(\be,n)$ with \al using \eqref{eq:n,be} and we put
\begin{equation}
s=\prod_{i\in I}y_i^{\rk\lP_i},\qquad Q^\be=\prod_{i\in I}y_i^{-c_1(\lP_i)\cdot\be}.
\label{eq:identification}
\end{equation}

%% file: s02.tex
\section{McKay correspondence}
\label{sec:mckay}
In this section we will consider one particular example of the general framework studied in section \ref{sec:perverse}. This example comes from the McKay correspondence. More precisely, we will study crepant resolutions of the singular affine $3$-Calabi-Yau varieties $\cC^2/G\xx\cC$, where $G\sb\SL_2(\cC)$ is a finite subgroup.

\subsection{McKay quiver}
Let $G$ be a finite group and let $V$ be a finite-dimensional $G$-representation over \cC. We define the McKay quiver $Q$ of the pair $(G,V)$ as follows. The set of vertices is given by the set $\hat G$ of irreducible representations of $G$. For any $\si,\rho\in\hat G$ the set of arrows from \si to \rho is given by a basis of
$$\Hom_G(\si,\rho\ts V).$$

Let $S=\cC[V]=SV^*$ be a symmetric algebra over $V^*$. Then any module over the skew group algebra $S\rtimes G$ induces a representation of the McKay quiver $Q$, see \eg \cite[Section 5.1]{mozgovoy_crepant}. More precisely, there is a full and faithful functor
$$\mod (S\rtimes G)\to\mod\cC Q.$$

\subsection{Classical Mckay correspondence}
\label{sec:classical}
Let $G$ be a finite subgroup of $\SL_2(\cC)$. We will denote $\cC^2$ by $V$. Let $X=V/G$ and let $f:Y\to X$ be its crepant resolution, where $Y=\Hilb^G(V)$ parametrizes $G$-invariant quotients of $\lO_V$ isomorphic to the regular $G$-representation $\cC G$. We define $S=\cC[V]=SV^*$ and $R=\cC[X]=S^G$.
Let $p_Y:Y\xx V\to Y$, $p_V:Y\xx V\to V$ be projections.

%and $p:Z\to Y$, $q:Z\to V$ be projections. 
%There is a commutative diagram
%\begin{diagram}
%Z&\rTo^q&V\\
%\dTo^p&&\dTo_g\\
%Y&\rTo^f&X
%\end{diagram}

Let $Z\sb Y\xx V$ be the tautological scheme and let $\lP'=p_{Y*}\lO_Z$ be the universal bundle on $Y=\Hilb^G(V)$. This is a bundle of $G$-representations isomorphic to the regular representation $\cC G$. We can decompose $\lP'=\bop_{\rho\in\hat G}\lP_\rho\ts\rho$, where $\rk\lP_\rho=\dim\rho$ (in the notation of Section \ref{sec:triples}, the index set $I$ is identified with $\hat G$ and $0\in I$ is identified with the trivial representation $\rho_0\in\hat G$).
The following result was proved in \cite{kapranov_kleinian}.

\begin{thr}
The functors
$$\Phi':D^b(\Coh Y)\to D^b(\Coh_GV),\qquad \Psi':D^b(\Coh_GV)\to D^b(\Coh Y)$$
defined by
\begin{align*}
\Phi'(F)=&Rp_{V*}R\lHom(\lO_Z,p_Y^!F),\\
\Psi'(F)=&(Rp_{Y*}(p_V^*F\lts\lO_Z))^G
%=(Rp_*(Lq^*F))^G
\end{align*}
are mutually inverse equivalences of categories.
\end{thr}

\begin{rmr}
One can see from the definition of functors $\Phi'$ and $\Psi'$ that $\Phi'$ is right adjoint to $\Psi'$. Note that
$p_Y^!(-)=p_Y^*(-)\ts p^*_V\om_V[2]=p_Y^*(-)[2]$. This shift is missing in \cite{kapranov_kleinian}.
\end{rmr}

\begin{rmr}
\label{brav}
Let $V$ be an arbitrary finite-dimensional $G$-representation.
The category $\Coh_GV$ can be identified with the category $\mod A$ of left finitely-generated $A$-modules, where $A=\cC[V]\rtimes G$. Consider a $\cC G$-module $W=\bop_{\rho\in\hat G}\rho$ and an $A$-module $\wtl W=A\ts_{\cC G}W\iso\cC[V]\ts W$. Define
$$J=\End_A(\wtl W)\op\iso(\cC[V]\ts \End_{\cC}(W)\op)^G.$$
Then the functor
$$\Hom_A(\wtl W,-):\mod A\to \mod J$$
is an equivalence of categories \cite[Theorem 2.1]{brav_projective}. Note that for any $A$-module $M$ we have $\Hom_A(\wtl W,M)\iso\Hom_{\cC G}(W,M)\iso\oplus_{\rho\in\hat G}M_\rho$, the last isomorphism being an isomorphism of vector spaces. In particular $J\iso\Hom_A(\wtl W,\wtl W)\iso \cC[V]\ts(\bop_{\rho\in\hat G}\cC)$ as a vector space.
Therefore $A\iso J$ if and only if $G$ is abelian.
\end{rmr}

\begin{rmr}
\label{rmr:Y and J}
Considering $\Phi'(F)$ as a complex over $A$ we obtain
%The functor \Phi in the above theorem can be interpreted as a functor $\Phi:D^b(\Coh Y)\to D^b(\mod J)$ given by
\begin{multline*}
\Phi'(F)=\RHom_{Y\xx V}(\lO_Z,p_Y^!F))
=R\Ga(Y,Rp_{Y*}R\lHom_{Y\xx V}(\lO_Z,p_Y^!F))\\
=R\Ga(Y,R\lHom_Y({p_{Y*}}\lO_Z,F)=\RHom_Y(\lP',F).
\end{multline*}
In particular $\Phi'(\lP')=\RHom_Y(\lP',\lP')\iso\cC[V]\ts \cC G$ (see \cite[Prop.~1.5]{kapranov_kleinian}). Then for $\lP=\bop_{\rho\in\hat G}\lP_\rho$ we have $\Phi'(\lP)=\cC[V]\ts\bop_{\rho\in\hat G}\rho=\wtl W$. Therefore the functor
$$\Phi=\RHom_A(\wtl W,-)\circ\Psi':D^b(Y)\to D^b(\mod J)$$
is given by
$$\Phi(F)=\RHom_A(\Phi'(\lP),\Phi'(F))=\RHom_Y(\lP,F)$$
and we are in the situation discussed in section \ref{sec:perv}.
%According to \cite[Section 2.3]{kapranov_kleinian} we have $\Phi(\lP_\rho)=\wtl\rho=\cC[V]\ts\rho$. This implies $\Phi(\lP)=\wtl W$ and $J\iso\End_Y(\lP)\op$. 
\end{rmr}

It is known that the McKay quiver of $(G,V)$ is a double quiver $\ub Q$ of an affine quiver $Q$. The algebra $A=\cC[V]\rtimes G$ is Morita-equivalent to the preprojective algebra $\Pi_Q$ of the quiver $Q$ \cite{crawley-boevey_noncommutative,lenzing_curve,reiten_two-dimensional} (recall that $\Pi_Q$ is a quotient of $\cC\ub Q$ by the relation $\sum_{a\in Q_1}(aa^*-a^*a)$). More precisely, the algebra $J$ considered in Remark \ref{brav} is isomorphic to $\Pi_Q$ by \cite[Remark 2.7]{brav_projective}.

Recall that $\Hilb^GV$ parametrizes $G$-invariant quotients of $\lO_V$ isomorphic to $\cC G$ as $G$-representations. Equivalently, it parametrizes pairs $(M,m)$, where $M$ is a $\cC[V]\rtimes G$-module isomorphic to $\cC G\iso\bop_{\rho\in\hat G}\rho^{\dim\rho}$ as a $G$-representation and $m\in M^G$ generates $M$. This implies that $\Hilb^GV$ can be identified with the moduli space of \ze-stable $\Pi_Q$-modules having dimension vector
\begin{equation}
\de=(\dim\rho)_{\rho\in\hat G},
\label{eq:de}
\end{equation}
with stability parameter $\ze$ given by
\begin{equation}
\ze_0=\sum_{\rho\in\hat G_*}\dim\rho=\rk\lP_*,\qquad \ze_i=-1,\ i\in\hat G_*,
\label{eq:ze}
\end{equation}
where $\hat G_*=\hat G\ms\set{\rho_0}$ and $\rho_0$ denotes the trivial representation.

%\begin{rmr}
%Note that the stability parameter \ze coincides with the stability parameter used in Remark \ref{zeta-explanation} for the definition of DT/PT invariants. I don't know if there exists a conceptual explanation of this phenomenon.
%\end{rmr}

It is shown in \cite{gonzalez-sprinberg_construction} that $(c_1(\lP_\rho))_{\rho\in\hat G_*}$ is a basis of $\Pic(Y)=H^2(Y,\cZ)$ dual to the basis of $H_2(Y,\cZ)$ given by irreducible components of $f\inv(0)$. This is the content of the classical McKay correspondence -- the bijection between the irreducible components of $f\inv(0)$ and the non-trivial irreducible representations of $G$, or equivalently, the vertices of the McKay quiver. One denotes by $C_\rho$ the irreducible component of $f\inv(0)$ dual to $c_1(\lP_\rho)$, $\rho\in\hat G_*$.

%It is shown in \cite[Prop.~1.2]{gonzalez-sprinberg_construction} that there is an isomorphism $$K_0(Y)\to\cZ\oplus\Pic(Y)=H^0(Y,\cZ)\oplus H^2(Y,\cZ),\qquad F\mto(\ch_0 F,\ch_1F).$$
%
%Let $M$ be a zero-dimensional sheaf over $X$ with a support in the nonsingular locus of $X$. Then $M$ can be considered as a sheaf over $Y$ and we can consider the object $\Phi(M)\in D^b(\mod A)$. We have
%$\Phi_A(M)=S\ts_RM$ and 
%$$\Phi_J(M)=\Hom_A(S\ts W,S\ts_RM)$$
%
%$$\Phi_J(M)=\Hom_Y(\lP,M)=\Hom_R(R\ts W,M)
%=\Hom_\cC(W,M)
%$$
%If $M'$ then
%$$\Phi_J(M)=\Hom_A(S\ts W,M')=\Hom(S\ts_R\ts R\ts W,M')=
%\Hom_R(R\ts W,\Hom(S,M'))
%$$
%Therefore $M'=S\ts_R M$.

\subsection{Dimension \texorpdfstring{$3$}{3} McKay correspondence}
\label{sec:dim3}
Let $G$ be a finite subgroup of $\SL_2(\cC)$ as before.
We can embed $G$ into $\SL_3(\cC)$ by letting $G$ act trivially on the third coordinate.
Let $X=V/G=(\cC^2/G)\xx\cC$, where $V=\cC^3$. It is proved in \cite[Theorem 1.2]{bridgeland_mukai}
that $Y=\Hilb^G(V)\iso\Hilb^G(\cC^2)\xx\cC$ is irreducible and the natural morphism $f:Y\to X$ is a crepant resolution. Here $\Hilb^G(V)$ parametrizes the $G$-invariant quotients of $\lO_V$ isomorphic to the regular $G$-representation $\cC G$.

The McKay quiver $\what Q$ of $(G,\cC^3)$ is obtained from $\ub Q$ (the McKay quiver of $(G,\cC^2)$) by adding loops $l_i:i\to i$ for every vertex $i\in Q_0$. By the classical McKay correspondence there is a bijection between $\hat G_*$ and irreducible components of $f\inv(0)$. Moreover, these irreducible components form a basis of $H_2(Y,\cZ)$.

Suppose that the quiver $Q$ mentioned earlier is an affine quiver of type $X_{l}^{(1)}$, $X=ADE$ ($l$ equals the number of elements in $\hat G_*$). Let \g be the corresponding affine Lie algebra. Its simple roots are in bijection with $\hat G$, \ie with irreducible representations of $G$.
%We will denote them by $(\al_\rho)_{\rho\in\hat G}$.
%We will assume that the root $\al_0$ corresponds to the trivial representation.

%It is known that imaginary root \de is given by
%$$\de=\sum_{\rho\in\hat G}\dim\rho\cdot\al_\rho.$$
%
%The points in $\Hilb^G(V)$ can be interpreted as representations of the skew group algebra $\cC[V]\rtimes G$.

As in Remark \ref{brav}, let $A=\cC[V]\rtimes G$, $W=\bop_{\rho\in\hat G}\rho$, and $\wtl W=A\ts_{\cC G}W$. The algebra $J=\End_A(\wtl W)\op$ is Morita equivalent to $A$. This algebra can be realized as a Jacobian algebra $J_{\what Q,W}$ of the McKay quiver $\what Q$ with a potential
\begin{equation}
W=\sum_{(a:i\to j)\in Q_1}(aa^*l_j-a^*al_i),
\label{eq:potential}
\end{equation}
where $a^*:j\to i$ is an arrow in $\ub Q$ dual to $a:i\to j$ and $l_i$ is a loop at vertex $i$. Then $\Hilb^G(V)$ can be identified with the moduli space $M_\ze(J,\de)$ of representations of $J$ having dimension vector \de from \eqref{eq:de} and semistable with respect to the stability parameter \ze from \eqref{eq:ze}.
Let \lP be the universal bundle on $M_\ze(J,\de)$. By the $3$-dimensional McKay correspondence proved in \cite{bridgeland_mukai,vandenbergh_three-dimensional} (to be more precise, in \cite{bridgeland_mukai} one constructs equivalence between $D^b(Y)$ and $D^b(\mod A)$ which can be interpreted as an equivalence between $D^b(Y)$ and $D^b(\mod J)$ similarly to Remark \ref{rmr:Y and J}) there is an equivalence of categories
$$\Phi:D^b(Y)\to D^b(\mod J),\qquad F\mto\RHom_Y(\lP,F).$$
As we have seen in Section \ref{DT/PT as reps}, in order to compute the DT/PT invariants of $Y$, we have to compute invariants of the moduli stacks of representations of some algebra $J'$ related to $J$. This will be done in the next section. 

\begin{rmr}
\label{special w}
Note that the algebra $J'$ defined in Section \ref{sec:triples} is obtained from the Jacobian algebra $J=J_{\what Q,W}$ by adding one new vertex $\infty$ to $\what Q$ and one new arrow $\infty\to0$, and considering $W$ as a potential for this new quiver. This coincides with the construction of $J'$ discussed in Section \ref{sec:framed} with the vector $w\in\cN^{Q_0}$ given by $w_0=1$ and $w_i=0$ for $i\ne0$.
\end{rmr}

%% file: s03.tex
\section{Loop double quivers and their DT invariants}
\label{sec:loop}
In the previous section, given a quiver $Q$ of affine type corresponding to a finite group $G\sb\SL_2(\cC)$, we have constructed a new quiver $\what Q$ with a potential and observed that we have to understand the moduli spaces of representations of the corresponding Jacobian algebra in order to understand Pandharipande-Thomas or Donaldson-Thomas invariants of $\Hilb^G(\cC^3)$. In this section we will start with an arbitrary quiver $Q$, construct canonically a new quiver $\what Q$ with a potential, and study the Donaldson-Thomas invariants of the corresponding Jacobian algebra. Clearly we will get much more than we actually need for the McKay case. The general situation turns out to be of independent interest.

Let $Q$ be a quiver. Define a double quiver $\ub Q$ to be obtained from $Q$ by adding arrows $a^*:j\to i$ for every arrow $(a:i\to j)\in Q_1$. Define a loop double quiver $\what Q$ (\cf \cite[Section 4.2]{ginzburg_calabi-yau}) to be obtained from $\ub Q$ by adding loops $l_i:i\to i$ for every vertex $i\in Q_0$.
We have
$$\hi_{\what Q}(\al,\be)=\hi_Q(\al,\be)+\hi_Q(\be,\al)-2\al\cdot\be.$$
%In particular
%$$T_{\tl Q}(\al)=2T_Q(\al)-2\sum_{i\in Q_0}{\al_i^2}.$$
Consider the following potential on $\what Q$
$$W=\sum_{(a:i\to j)\in Q_1}(aa^*l_j-a^*al_i).$$
There exists a cut $I\sb\what Q_1$ of $(\what Q,W)$ consisting of loops $l_i$, $i\in Q_0$.

Let $A_U$ be the Donaldson-Thomas series of $(\what Q,W)$ and let $\Om_\al\in\lvc$, $\al\in\cN^{Q_0}$, be its Donaldson-Thomas invariants defined in \eqref{eq:DT}
(note that $\what Q$ is a symmetric quiver).

For any $\al\in\cN^{Q_0}$ let $a_\al(q)$ be the polynomial counting absolutely indecomposable representations of $Q$ of dimension \al over finite fields \cite{kac_root}. It was proved by Kac \cite{kac_root} that this polynomial has integer coefficients and it was conjectured in \cite{kac_root} that it has non-negative coefficients. Let $\Pi_Q$ be the preprojective algebra of the quiver $Q$.

\begin{thr}
\label{u motive}
We have
$$A_U
=\sum_{\al\in\cN^{Q_0}}\cL^{\hi_{Q}(\al,\al)}\frac{[R(\Pi_Q,\al)]}{[\GL_\al]}y^\al
=\Exp\left(\frac{\sum_{\al\in\cN^{Q_0}}a_\al(\cL) y^\al}{1-\cL\inv}\right).$$
Motivic DT invariants (see \eqref{eq:DT}) are given by $\Om_\al=a_\al(\cL)$, $\al\in\cN^{Q_0}$.
\end{thr}
\begin{proof}
For any $\al\in\cN^{Q_0}$, let $\Ind_\al$ be the set of indecomposable representations of $Q$ of dimension \al. It is known that $\Ind_\al$ can be decomposed (non-canonically) into a finite union of algebraic varieties. The motive of $\Ind_\al$ is well-defined and equals $a_\al(\cL)$.
Let $\Ind=\coprod_{\al\in\cN^{Q_0}}\Ind_\al.$

To prove the statement of the theorem, we are going to apply Proposition \ref{prp:reduction}. Let
$$I=\sets{l_i}{i\in Q_0}\sb \what Q_1.$$
Then $\what Q_I=(Q_0,\what Q_1\ms I)=\ub Q$. The quotient algebra
$$J_{W,I}=\what Q_I/(\dd_IW)=(\dd W/\dd a\mid a\in I)$$
is just the preprojective algebra $\Pi_Q$ of $Q$, as for any $i\in Q_0$ we have
$$\dd W/\dd l_i=\sum_{t(a)=i}aa^*-\sum_{s(a)=i}a^*a.$$
Note that
$$d_I(\al)=\sum_{(a:i\to j)\in I}\al_i\al_j=\sum_{i\in Q_0}\al_i^2.$$
Applying the first dimensional reduction (Proposition \ref{prp:reduction}) we obtain
\begin{equation}
A_U=\sum_{\al\in\cN^{Q_0}}(-\cL^\oh)^{\hi_{\what Q}(\al,\al)+2d_I(\al)}
\frac{[R(J_{W,I},\al)]}{[\GL_\al]}y^\al
=\sum_{\al\in\cN^{Q_0}}\cL^{\hi_{Q}(\al,\al)}
\frac{[R(\Pi_Q,\al)]}{[\GL_\al]}y^\al.
\label{eq:A_U1}
\end{equation}
Now we perform the second dimension reduction -- go from the algebra $\Pi_Q$ of homological dimension $2$ to the algebra $\cC Q$ of homological dimension $1$.
To do this we take a different cut $I'=\sets{a^*}{a\in Q_1}$ of $(\what Q,W)$ and consider the forgetful map from $R(J_{W,I},\al)$ to $R(\what Q_{I\cup I'},\al)$, where $\what Q_{I\cup I'}=(Q_0,\what Q_1\ms(I\cup I'))=Q$. Important point for us is that the fibers of this map are vector spaces and we can determine their dimensions. More precisely, we consider the forgetful map
$$R(\Pi_Q,\al)\to R(Q,\al).$$
By \cite[Lemma 4.2]{crawley-boevey_noncommutative} its fiber over a $Q$-representation $M$ can be identified with $\Ext^1(M,M)^*$. Let $M=\oplus_{X\in\Ind}X^{m(X)}$, where $m:\Ind\to\cN$ is a map with finite support. The contribution of the isomorphism class of $M$ to $\frac{[R(\Pi_Q,\al)]}{[\GL_\al]}$ is given by (see Theorem \ref{aut})
\begin{equation}
\frac{\cL^{\dim\Ext^1(M,M)}}{[\Aut M]}=
\frac{\cL^{\dim\Ext^1(M,M)-\dim\Hom(M,M)}}{\prod_{X\in\Ind} (\cL\inv)_{m(X)}}=
\frac{\cL^{-\hi_Q(M,M)}}{\prod_{X\in\Ind} (\cL\inv)_{m(X)}}.
\label{eq:contrib}
\end{equation}
It follows from \eqref{eq:A_U1} and \eqref{eq:contrib} that
\begin{align*}
A_U
&=\sum_{m:\Ind\to\cN}\prod_{X\in\Ind}\frac{y^{m(X)\dim X}}{(\cL\inv)_{m(X)}}\\
&=\prod_{\al\in\cN^{Q_0}}\sum_{m:\Ind_\al\to\cN}\prod_{X\in\Ind_\al}\frac{y^{m(X)\al}}{(\cL\inv)_{m(X)}}\\
&=\prod_{\al\in\cN^{Q_0}}\Pow\left(\sum_{m\ge0}\frac{y^{m\al}}{(\cL\inv)_m},[\Ind_\al]\right)
&\text{(Equation \eqref{power})}\\
&=\prod_{\al\in\cN^{Q_0}}\Pow\left(\Exp\left(\frac{y^\al}{1-\cL\inv}\right),a_\al(\cL)\right)
&\text{(Equation \eqref{eq:heine})}\\
&=\prod_{\al\in\cN^{Q_0}}\Exp\left(\frac{a_\al(\cL)y^\al}{1-\cL\inv}\right)\\
&=\Exp\left(\frac{\sum_{\al\in\cN^{Q_0}}a_\al(\cL)y^{\al}}{1-\cL\inv}
\right).
\end{align*}
\end{proof}

\begin{rmr}
Let $Q$ be a quiver with one vertex and one loop. In this case $R(\Pi_Q,n)$, for $n\ge1$, consists of pairs of commuting $n\xx n$ matrices. The polynomials $a_n(q)$ are given by $a_n(q)=q$, $n\ge1$, and $a_0(q)=0$. The above theorem says that 
$$\sum_{n\ge0}\frac{[R(\Pi_Q,n)]}{[\GL_n]}y^n
=\Exp\left(\frac{\cL}{1-\cL\inv}\sum_{n\ge1}y^n\right).$$
This is a classical result of Feit and Fine \cite{feit_pairs} which was used in \cite{behrend_motivic} to compute the motivic DT invariants of $\cC^3$.
\end{rmr}

\begin{rmr}
Non-negativity of the coefficients of the Donaldson-Thomas invariants is equivalent to the Kac positivity conjecture, which states that the coefficients of the polynomials $a_\al$ are non-negative.
\end{rmr}

\begin{conj}
\label{conj:nonneg}
Let $(Q,W)$ be a symmetric quiver with potential. Assume that $(Q,W)$ admits two cuts $I,I'$ such that $I\cap I'=\emptyset$. Then the Donaldson-Thomas invariants $\Om_\al$, $\al\in\cN^{Q_0}$, are polynomials in $\cL^{\pm\oh}$ and $\Om_\al(-\cL^\oh)$ are polynomials with integer, non-negative coefficients.
\end{conj}

In the case of a trivial potential this was conjectured by Kontsevich and Soibelman \cite{kontsevich_cohomological} and proved by Efimov \cite{efimov_cohomological}. As we have seen, in the case of $(\what Q,W)$ our conjecture is equivalent to the Kac positivity conjecture which was proved in many (but not all) cases \cite{crawley-boevey_absolutely,mozgovoy_motivic}. Our conjecture is also true for the quiver with potential for the conifold. This can be seen from the explicit formula \cite[Theorem 2.1]{morrison_motivica}.

\begin{rmr}
For any $\ze\in\cR^{Q_0}$ let $R_{\ze}(\Pi_Q,\al)\sb R(\Pi_Q,\al)$ be the space of \ze-semistable representations of $\Pi_Q$. Define the slope function $\mu_\ze(\al)=\frac{\ze\cdot\al}{\n\al}$ for $\al\in\cN^{Q_0}\ms\set0$, where $\n\al=\sum\al_i$. We can show, using the above theorem, that for any $\mu\in\cR$
$$\sum_{\mu_\ze(\al)=\mu}\cL^{\hi_{Q}(\al,\al)}\frac{[R_\ze(\Pi_Q,\al)]}{[\GL_\al]}y^\al
=\Exp\left(\frac{\sum_{\mu_\ze(\al)=\mu}a_\al(\cL) y^\al}{1-\cL\inv}\right).$$
For indivisible \al this implies
$$a_\al(\cL)=
\cL^{\hi_{Q}(\al,\al)-1}\frac{(\cL-1)[R_\ze(\Pi_Q,\al)]}{[\GL_\al]}.$$
The last equation was proved earlier in \cite[Prop.~2.2.1]{crawley-boevey_absolutely}.
\end{rmr}

%% file: s04.tex
\section{Motives for the McKay quiver}
\label{sec:applications}
Let $G\sb\SL_2(\cC)$ be a finite subgroup, $\what Q$ be the McKay quiver of $(G,\cC^3)$ as in Section \ref{sec:dim3}, $W$ be the potential for $\what Q$ from \eqref{eq:potential}, and let $Q$ be the corresponding affine quiver of type $X_l^{(1)}$ ($X=ADE$), where $l=\n{\hat G_*}$. Let $\De$ be the set of roots of the affine Lie algebra of type $X_l^{(1)}$ and let $\Deo$ be the set of roots of the finite Lie algebra of type $X_l$.
The vector
$$\de=(\dim\rho)_{\rho\in\hat G}$$
defined in \eqref{eq:de} is the indivisible imaginary root of $\De_+$.
Positive real roots in \De can be described as
$$\De_+^\re=\sets{\al+n\de}{\al\in\Deo,n\ge1}\cup\Deop.$$
We can decompose $\De_+^\re=\De_+^{\re,+}\cup\De_+^{\re,-}\cup
\Deop$, where
$$\De_+^{\re,+}=\sets{\al+n\de}{\al\in\Deop,n\ge1},\qquad
\De_+^{\re,-}=\sets{\al+n\de}{\al\in\Deom,n\ge1}.$$
Positive imaginary roots in \De can be described as
$$\De_+^\im=\sets{n\de}{n\ge1}.$$

\begin{thr}
The universal Donaldson-Thomas series of $(\what Q,W)$ equals
$$A_U=\Exp\left(\frac{\sum_{\al\in\De_+^\re}y^\al+(\cL+l)\sum_{\al\in\De_+^\im}y^\al}{1-\cL\inv}\right).$$
\end{thr}
\begin{proof}
Let \g be the affine Lie algebra of type $X_l^{(1)}$.
The motive of indecomposable representations of $Q$ of dimension $n\de$ equals
$\cL+\dim\g_{n\de}$. It is known (see \eg \cite{kac_infinite-dimensional})
%Remark 12.13 
that $\dim\g_{n\de}=l$.
Now we apply Theorem \ref{u motive}.
\rem{For example for the Jordan quiver $A_0^{(1)}$ we get motive \cL}
\end{proof}

According to Corollary \ref{crl:Z factoriz}, for any generic $\ze\in\cR^{Q_0}$, the generating function of virtual motives of \ze-stable framed objects is
$$\lZ_\ze=S_{-w}\Exp\left(\sum_{\ze\cdot\al<0}\frac{\cL^{w\cdot\al}-1}{1-\cL\inv}\Om_\al y^\al\right).$$
where
$w\in\cN^{Q_0}$ is given as in Remark \ref{special w} by $w_0=1$ and $w_i=0$, $i\ne0$ and $\Om_\al=1$ for $\al\in\De_+^\re$ and $\Om_\al=\cL+l$ for $\al\in\De_+^\im$.

\begin{crl}
\label{crl:factorization}
We have
$$\lZ_\ze=\prod_{\ze\cdot\al<0}\lZ_\al$$
where
\begin{equation*}
\lZ_\al(-y_0,y_1,\dots)=\case{
\prod_{j=1}^{\al_0}(1-\cL^{j-\frac{\al_0}2}y^\al)\inv
&\al\in\De_+^\re
\\
\prod_{j=1}^{\al_0}(1-\cL^{j+1-\frac{\al_0}2}y^\al)\inv(1-\cL^{j-\frac{\al_0}2}y^\al)^{-l}
&\al\in\De_+^\im
\\
1&\text{otherwise}}
\end{equation*}
Note that $\lZ_\al=1$ for $\al\in\Deop$. 
\end{crl}
\begin{proof}
We have
$$\lZ_\ze=\prod_{\ze\cdot\al<0} S_{-w}\Exp\left(\frac{\cL^{w\cdot\al}-1}{1-\cL\inv}\Om_\al y^\al\right).$$
Note that $w\cdot\al=\al_0$ and therefore $S_{-w}$ acts as $y^\al\mto\cL^{-\frac{\al_0}2}y^\al$. For $\Om_\al=1$ or $\Om_\al=\cL+l$ we have
$$\Exp\left(\frac{\cL^{\al_0}-1}{1-\cL\inv}y^\al\right)
=\prod_{j=1}^{\al_0}(1-\cL^jy^\al)\inv,$$
$$\Exp\left(\frac{(\cL^{\al_0}-1)(\cL+l)}{1-\cL\inv}y^\al\right)
=\prod_{j=1}^{\al_0}(1-\cL^{j+1}y^\al)\inv(1-\cL^jy^\al)^{-l}.$$
Now we apply $S_{-w}$ to both equations.
\end{proof}

Define $r=\sum_{\rho\in\hat G_*}\dim\rho=\n\de-1$ and consider the following stability parameters
$$\ze=(-r,1,\dots,1)\in\cR^{Q_0},\qquad \ze^{\pm}=(-r\pm\eps,1,\dots,1)\in\cR^{Q_0}.$$
We will often write $\lZ_{PT}(s,Q)$ and $\lZ_{DT}(s,Q)$ for $\lZ_{PT}(Y,s,Q)$ and $\lZ_{DT}(Y,s,Q)$.
It follows from \eqref{eq:gDT}, \eqref{eq:PT} that
\begin{equation}
\lZ_{PT}(-s,Q)=\lZ_{\ze^+}(-y_0,y_1,\dots),\qquad \lZ_{DT}(-s,Q)=\lZ_{\ze^-}(-y_0,y_1,\dots),
\end{equation}
where
\begin{equation}
s=\prod_{i\in I}y_i^{\rk\lP_i}=y^\de,\qquad Q^\be=\prod_{i\in I}y_i^{-c_1(\lP_i)\cdot\be}=y^{-\be}.
\end{equation}

\begin{lmm}
We have $\sets{\al\in\De_+}{\ze^\pm\cdot\al=0}=\emptyset$ and
$$\sets{\al\in\De_+}{\ze^+\cdot\al<0}=\De^{\re,-}_+\qquad \sets{\al\in\De_+}{\ze^-\cdot\al<0}=\De^{\re,-}_+\cup\De^\im_+.$$
\end{lmm}
\begin{proof}
Our statement follows from the fact that
$$\sets{\al\in\De_+}{\ze\cdot\al=0}=\De^\im_+,\qquad
\sets{\al\in\De_+}{\ze\cdot\al<0}=\De^{\re,-}_+.$$
\end{proof}

\begin{crl}
\label{crl:PT/DT}
We have
\begin{align*}
\lZ_{PT}(-s,Q)
&=\prod_{n\ge1}\prod_{\al\in\Deo_+}\Exp\left(\frac{\cL^n-1}{\cL-1}\cL^{1-\frac n2}s^nQ^\al\right),\\
\lZ_{DT}(-s,Q)
&=\lZ_{PT}(-s,Q)\cdot
\prod_{n\ge1}\Exp\left(\frac{(\cL^n-1)(\cL+l)}{\cL-1}\cL^{1-\frac n2}s^n\right).
\end{align*}
\end{crl}
\begin{proof}
We have
\begin{multline*}
\lZ_{PT}(-s,Q)
=\lZ_{\ze^+}(-y_0,y_1,\dots)
=\prod_{\al\in\De_+^{\re,-}}
\Exp\left(\frac{\cL^{\al_0}-1}{\cL-1}\cL^{1-\frac{\al_0}2}y^\al\right)\\
=\prod_{n\ge1}\prod_{\al\in\Deo_-}\Exp\left(\frac{\cL^n-1}{\cL-1}\cL^{1-\frac n2}y^{\al+n\de}\right)
=\prod_{n\ge1}\prod_{\al\in\Deo_+}\Exp\left(\frac{\cL^n-1}{\cL-1}\cL^{1-\frac n2}s^nQ^\al\right).
\end{multline*}
The proof for $\lZ_{DT}$ is similar.
\end{proof}

\begin{crl}
\label{expl DT/PT}
We have
\begin{align*}
\lZ_{PT}(-s,Q)
&=\prod_{n\ge1}\prod_{j=1}^n\prod_{\al\in\Deo_+}(1-\cL^{j-\frac n2}s^nQ^\al)\inv,\\
\lZ_{DT}(-s,Q)
&=\lZ_{PT}(-s,Q)\cdot\prod_{n\ge1}\prod_{j=1}^{n}(1-\cL^{j+1-\frac{n}2}s^n)\inv(1-\cL^{j-\frac{n}2}s^n)^{-l}.
\end{align*}
\end{crl}

\begin{rmr}
It was proved in \cite{behrend_motivic} that
$\lZ_{Y}(s)=\sum_{n\ge0}[\Hilb^n Y]_\vir s^n$ is given by
$$\lZ_Y(-s)=\Pow(\lZ_{\cC^3}(-s),\cL^{-3}[Y]),$$
where
$$\lZ_{\cC^3}(-s)=\prod_{n\ge1}\prod_{j=1}^n(1-\cL^{j+1-\frac n2}s^n)\inv
=\Exp\left(\sum_{n\ge1}\frac{\cL^n-1}{\cL-1}\cL^{2-\frac n2}s^n\right).$$
In our case $[Y]=\cL^3+l\cL^2$. Therefore
$$\lZ_Y(-s)
=\Exp\left((1+l\cL\inv)\sum_{n\ge1}\frac{\cL^n-1}{\cL-1}\cL^{2-\frac n2}s^n\right)
=\lZ_{DT}(-s,Q)/\lZ_{PT}(-s,Q).$$
The property $\lZ_{DT}(s,Q)=\lZ_{PT}(s,Q)\lZ_Y(s)$ is expected to be true for general $3$-Calabi-Yau manifolds \cite[Remark 4.4]{morrison_motivica}.
This property for numerical invariants, called DT/PT correspondence, was conjectured in \cite[Conj.3.3]{pandharipande_curve} and proved in \cite{bridgeland_hall,toda_curve}.
\end{rmr}

\subsection{Classical limit}
The classical limit $\ub\lZ_{PT}$ of $\lZ_{PT}$ (resp.\ $\ub\lZ_{DT}$ of $\lZ_{DT}$) is obtained by taking the Euler number specialization $\cL^\oh\mto1$.

Define the generalized MacMahon function
$$M(x,q)=\prod_{n\ge1}(1-xq^n)^{-n}=\Exp\bigg(\sum_{n\ge1}nxq^n\bigg)=\Exp\left(\frac{xq}{(1-q)^2}\right)$$
and $M(q)=M(1,q)=\prod_{n\ge1}(1-q^n)^{-n}$.

The following result was proved in \cite{gholampour_counting} using localization techniques (we use variable $q$ instead of variable $s$ used earlier for historical reasons).

\begin{crl}
We have
$$\ub\lZ_{PT}(-q,Q)=\prod_{\be\in\Deo_+}\prod_{n\ge1}(1-q^nQ^\be)^{-n}
=\prod_{\be\in\Deo_+}M(Q^\be,q),$$
$$\ub\lZ_{DT}(-q,Q)
=\ub\lZ_{PT}(-q,Q)\prod_{n\ge1}(1-q^n)^{-(l+1)n}
=\ub\lZ_{PT}(-q,Q)M(q)^{l+1}.$$
%$\ub\lZ_{NCDT}=M(-y^\de)^{l+1}\prod_{\be\in\Deo}M(y^\be,-y^\de)$.
\end{crl}
\begin{proof}
Just apply Corollary \ref{expl DT/PT}.
\end{proof}

\begin{rmr}
\label{rmr:NCDT formula}
According to Remark \ref{rmr:NCDT} and Corollary \ref{crl:factorization}, we can write the generating function of motivic NCDT invariants as $\lZ_{NCDT}=\prod_{\al\in\De_+}\lZ_\al$. Specialization of this formula at $\cL^\oh=1$ gives numerical NCDT invariants
$$\ub\lZ_{NCDT}(-q,Q)
=\prod_{\be\in\Deo}\prod_{n\ge1}(1-q^nQ^\be)^{-n}\cdot
\prod_{n\ge1}(1-q^n)^{-(l+1)n}
=M(q)^{l+1}\prod_{\be\in\Deo}M(Q^{\be},q).$$
This result was obtained earlier by Gholampour and Jiang \cite[Theorem 1.7]{gholampour_counting}. For abelian $G$ this result was proved by Young \cite[Theorem 1.4]{young_generating} using combinatorics and by Nagao \cite[Theorem 2.20]{nagao_derived} using wall-crossing formulas.
\end{rmr}

\subsection{Gopakumar-Vafa invariants}
It follows from the GW/DT/PT correspondence \cite[Conj.3.3]{pandharipande_curve} (see also \cite[Conj.3]{maulik_gromov-witten}) that
\begin{equation}
\lZ'_{GW}(Y,\la,Q)=\exp\left(\sum_{\be\ne0}\sum_{g\ge0}N_{g,\be}\la^{2g-2}Q^\be\right)=\ub \lZ_{PT}(Y,-q,Q),
\end{equation}
where $N_{g,\be}$ are the Gromov-Witten invariants of $Y$ and where we identify $q=e^{i\la}$. This determines the Gromov-Witten invariants. Their direct computation can be found in \cite{gholampour_counting}.
The Gopakumar-Vafa invariants $n_{g,\be}$ are determined by the formula~\cite{hosono_relative}
\begin{equation}
\lZ'_{GW}(Y,\la,Q)
=\exp\bigg(\sum_{\be\ne0}\sum_{g\ge0,k\ge1}\frac1k\Big(2\sin\frac{k\la}2\Big)^{2g-2}n_{g,\be}Q^{k\be}\bigg).
\end{equation}
%where we identify $s=e^{i\la}$. 
Equivalently (under the GW/PT correspondence)
\begin{equation}
\ub\lZ_{PT}(-q,Q)=\Exp\bigg(\sum_{\be\ne0}\sum_{g\ge0}(2-q-q\inv)^{g-1}n_{g,\be} Q^{\be}\bigg).
\label{eq:}
\end{equation}
We have seen that
$$\ub\lZ_{PT}(-q,Q)=\prod_{\be\in\Deo_+} M(Q^\be,q)
=\prod_{\be\in\Deo_+} \Exp\left(\frac{qQ^\be}{(1-q)^2}\right)
=\Exp\left(\frac{\sum_{\be\in\Deo_+} Q^\be}{q+q\inv-2}\right).$$
This implies $n_{g,\be}=-1$ for $g=0$ and $\be\in\Deop$ and zero otherwise (see \cite[Cor.~1.6]{gholampour_counting}).